\documentclass{amsproc}
\usepackage{amssymb}
\usepackage{graphicx}
\DeclareMathOperator{\ad}{ad}

\DeclareMathOperator{\Id}{id}

\DeclareMathOperator{\Der}{Der}

\newcommand{\F}{\mathbb F}

\newcommand{\Q}{\mathbb Q}

\newcommand{\LL}{\mathcal L}
\newcommand{\Z}[0]{\mathbb Z}

\newtheorem{dummy}{Dummy}

\newtheorem{theorem}[dummy]{Theorem}

\newtheorem{prop}[dummy]{Proposition}

\theoremstyle{definition}

\newtheorem{example}[dummy]{Example}

\theoremstyle{remark}

\begin{document}

\title[Grading switching for modular non-associative algebras]{Gradings switching for modular non-associative algebras}

\author{Marina Avitabile}
\email{marina.avitabile@unimib.it}
\address{Dipartimento di Matematica e Applicazioni\\
  Universit\`a degli Studi di Milano - Bicocca\\
 via Cozzi 53\\
  I-20125 Milano\\
  Italy}

\author{Sandro Mattarei}
\email{mattarei@science.unitn.it}
\address{Dipartimento di Matematica\\
  Universit\`a degli Studi di Trento\\
  via Sommarive 14\\
  I-38050 Povo (Trento)\\
  Italy}
\date{}

\begin{abstract}
We describe a {\em grading switching} for arbitrary non-associative
algebras of prime characteristic $p$, aimed at producing a new grading of an algebra from a given one.
This is inspired by a fundamental tool in the classification theory of modular Lie algebras known as {\em toral
switching,} which relies on a delicate adaptation of the exponential of a derivation.
We trace the development of grading switching, from an early version based on
taking the Artin-Hasse exponential of a nilpotent derivation, to a more general version
which uses certain generalized Laguerre polynomials playing the role of generalized exponentials.
Both versions depend on the existence of appropriate analogues
of the functional equation $e^x\cdot e^y=e^{x+y}$ for the classical exponential.
\end{abstract}

\maketitle

\section{Introduction}

The exponential function plays a role in various branches of mathematics.
The main reason, sometimes disguised in other forms, such as its differential
formulation
$(d/dx)e^x=e^x$, is that it
interconnects additive and multiplicative structures,
because of the fundamental identity $e^x\cdot e^y=e^{x+y}$.
In particular, one of the important classical applications is the local
reconstruction of a Lie group from its Lie algebra.
This Lie-theoretic use of the exponential function can be
formulated in more general terms as a device which turns derivations
of a {\em non-associative} (in the standard meaning of {\em not necessarily associative}) algebra into automorphisms.
The basic algebraic fact is already visible in the special case of nilpotent derivations,
where convergence matters play no role:
if $D$ is a nilpotent derivation of a non-associative algebra $A$ over a field of characteristic zero,
then the finite sum $\exp(D)=\sum_{i=0}^{\infty}D^i/i!$ defines an
automorphism of $A$.

This very useful property breaks down over fields of positive
characteristic $p$.
The condition $D^p=0$, which seems the minimum requirement for $\exp(D)$ to make sense in
this context, does not guarantee that $\exp(D)$ is an automorphism.
In fact, only the stronger assumption $D^{(p+1)/2}=0$ does, for
$p$ odd.
In the absence of the assumption $D^p=0$ one can use the {\em truncated exponential}
$E(X)=\sum_{i=0}^{p-1}X^i/i!$ as some kind of substitute for the
exponential series, of course dropping any expectation
that evaluating it on $D$ may yield an automorphism.

In the theory of modular Lie algebras the apparent shortcoming of $\exp(D)$
not necessarily being an automorphism when it is defined is
turned into an advantage with the technique of {\em toral
switching}.
This is a fundamental tool originally due to Winter~\cite{Win:toral},
which has undergone substantial generalizations by Block and Wilson~\cite{BlWil:rank-two}, and finally Premet~\cite{Premet:Cartan},
where maps similar to exponentials of derivations are used to
produce a new torus from a given one.
The very fact that the map need not be an automorphism
allows the new torus to have rather different
properties than the original one, which are more suited to classification purposes.

The function of tori in modular Lie algebras
is to produce gradings, as the corresponding eigenspace
decompositions with respect to the adjoint action (a {\em (generalized) root space decomposition}).
One naturally wonders whether some kind of exponential could be used
to pass from a grading to another without reference to the grading
arising as the root space decomposition with respect to some torus.
Our goal in this paper is to describe such a {\em grading switching}.
Besides effectively extending the applicability of the technique
from the realm of Lie algebras to the wider one of non-associative algebras,
it has applications within Lie algebra theory, where not all gradings of interest
are directly related to tori.
In particular, root space decompositions are gradings over abelian groups of exponent $p$,
while grading switching avoids this restriction.

Generally speaking, grading switching applies when
a graded algebra $A$ in characteristic $p$ has a homogeneous derivation $D$ of nonzero degree,
but such that $p$ times the degree of $D$ equals zero in the grading group.
For our scope a grading of $A$ is a direct sum decomposition
$A=\bigoplus_{g\in G} A_g$, where $G$ is an abelian group and
$A_gA_h\subseteq A_{g+h}$ holds (writing the group operation additively).

The simplest instance of grading switching occurs if the derivation satisfies $D^p=0$.
Then one easily finds that the truncated exponential $E(D)$ maps the given grading into another grading,
in the sense that $A=\bigoplus_{g\in G} E(D)A_g$,
as stated in Theorem~\ref{thm:truncated-grading}.

The assumption $D^p=0$ is very strong, and trying to relax it one finds that
the natural substitute for the exponential of a derivation in prime characteristic
appears to be the {\em Artin-Hasse exponential} $E_p(D)$
rather than the truncated exponential.
In fact,
it was shown in~\cite{Mat:Artin-Hasse} that in the above setting
$A=\bigoplus_{g\in G} E_p(D)A_g$
remains true with the condition $D^p=0$ weakened to the assumption that $D$ is nilpotent,
see our Theorem~\ref{thm:AH-grading}.

To attain a grading switching in full generality, that is, for arbitrary derivations $D$,
one needs a new substitute for the
exponential series, given by certain (generalized) {\em Laguerre polynomials}.
This was developed in~\cite{AviMat:Laguerre}, and may be thought as encompassing
some aspects of toral switching as a special case, when the grading under consideration
arises as a root space decomposition.
In fact, much inspiration for this work came from the very clear
exposition in Strade's book~\cite[Section~1.5]{Strade:book} of the most general version of toral switching, due to Premet~\cite{Premet:Cartan}.
In the general setting of grading switching under consideration,
and under very mild assumptions on $D$ (none in the finite-dimensional case
over an algebraically closed field)
those Laguerre polynomials allow one to construct a linear map
$\LL_D:A\to A$ such that $A=\bigoplus_{g\in G} \LL_D(A_k)$.
The map $\LL_D$ coincides with $E_p(D)$ on the Fitting null component of $D$, on which the latter makes sense.
We refer the reader to Section~\ref{sec:laguerre}
and Theorem~\ref{thm:Laguerre-grading} for details.

The development of the grading switching described here was motivated by applications
in the theory of {\em thin Lie algebras}.
The earlier version of grading switching based on Artin-Hasse exponentials was sufficient
for~\cite{AviMat:-1},
while an application in~\cite{AviMat:mixed_types} to a non-nilpotent derivation
has required the more general version based on Laguerre polynomials.

In Section~\ref{sec:applications} we present an application of grading switching
to the construction of certain new gradings of Zassenhaus and Albert-Zassenhaus algebras starting from natural ones.
Aside from special cases which we describe there,
those grading switchings do not arise as toral switchings.

\section{Artin-Hasse exponentials of derivations}\label{sec:artin-hasse}

It is well known that if $D$ is a nilpotent derivation of
a non-associative algebra $A$ over a field of
characteristic $0$, then $\exp(D)$ is an automorphism of $A$.
Here the exponential is defined by the ordinary exponential series $\exp(X)=\sum_{i=0}^{\infty}X^i/i!$,
and the assumption that $D$ is nilpotent is a convenient way of making sure that the series can be evaluated on it,
but can be weakened to suitable convergence assumptions.
The fact that $\exp(D)$ is an automorphism follows by simple calculation,
but in view of the variations and generalizations to follow it is best
deduced from the basic functional equation
$\exp(X)\cdot\exp(Y)=\exp(X+Y)$
of the exponential series, by means of the following {\em tensor product device}.
If $m: A\otimes A \rightarrow A$ denotes the map given
by the multiplication in $A$,
the fact that $D$ is a derivation means that
$D(m(x \otimes y))=m(Dx \otimes y)+m(x \otimes Dy)$ for any $x,y \in A$.
This property can be more concisely written as
$D\circ m=m \circ(D\otimes\Id+\Id\otimes D)$, where $\Id:A \rightarrow A$ is the
identity map.
Evaluating the functional equation of the exponential series on the commuting linear operators
$X=D\otimes\Id$ and $Y=\Id\otimes D$ now proves that $\exp(D)$ is an automorphism of $A$.

Moving on to the modular case, assume that
$A$ is a non-associative algebra over a field of prime
characteristic $p$, and $D$ is derivation of $A$.
Here the exponential series does not even make sense,
because the denominators vanish except for the first $p$ terms of the series.
One way to still make sense of the exponential, which has played an important role in the theory of modular
Lie algebras, is assuming that the derivation satisfies $D^p=0$, so that
$\exp(D)$ may be interpreted as $\sum_{i=0}^{p-1} D^i/i!$.
This interpretation is not without danger, because $\exp(D)$ need not be an automorphism if so interpreted.
To avoid ambiguities one better defines a {\em truncated exponential}
$E(D)=\sum_{i=0}^{p-1} D^i/i!$, which of course can be evaluated on an arbitrary derivation $D$.
Direct computation (possibly aided by the tensor product device if one so wishes) shows that
\begin{equation}\label{eq:obstruction}
E(D)x\cdot E(D)y-E(D)(xy)=
\sum_{k=p}^{2p-2}\sum_{i=k+1-p}^{p-1}
  \frac{(D^i x)(D^{k-i} y)}{i!(k-i)!}
\end{equation}
for $x,y\in A$.
If $p$ is odd and $D^{(p+1)/2}=0$, then at least one of the factors at the numerator of each summand
vanishes, and so $E(D)$ is an automorphism in this case.
It is well known that this need not be so in general, even if $D^p=0$, see~\cite[Section~5]{Mat:Artin-Hasse}, for example.
Indeed, the obstruction formula~\eqref{eq:obstruction}
gives a measure of the extent to which $E(D)$ fails to be an
automorphism.

Less well known is the following expression for the obstruction~\cite[Lemma~2.1]{Mat:Artin-Hasse}, which holds
under the assumption $D^p=0$:
\begin{equation}\label{eq:my-obstruction}
E(D)x\cdot E(D)y-E(D)(xy)=
E(D)
\sum_{i=1}^{p-1}\frac{(-1)^i}{i}\, D^ix\cdot D^{p-i}y
\end{equation}
for $x,y\in A$.
Because
$i!\,(p-i)!\equiv (-1)^i\, i\pmod{p}$
for $0<i<p$
(an easy consequence of Wilson's theorem
$(p-1)!\equiv -1\pmod{p}$),
the sum
$\sum_{i=1}^{p-1}(-1)^i\,D^ix\cdot D^{p-i}y/i$
which appears in the right-hand side of Equation~\eqref{eq:my-obstruction} can be rewritten as
\[
\sum_{i=1}^{p-1}\frac{D^ix\cdot D^{p-i}y}{i!\,(p-i)!}.
\]
Note that this sum consists of all terms of lowest degree in $D$ (that is, those corresponding to $k=p$)
in Equation~\eqref{eq:obstruction}.
According to the tensor product device introduced earlier,
Equation~\eqref{eq:my-obstruction} follows from the following polynomial congruence
by evaluating it on the commuting linear operators $X=D\otimes\Id$ and $Y=\Id\otimes D$.

\begin{prop}\label{prop:functional-truncated}
The truncated exponential polynomial $E(X)=\sum_{i=0}^{p-1}X^i/i!$
satisfies the congruence
\[
E(X)\cdot E(Y)\equiv
E(X+Y)
\Bigl(1+
\sum_{i=1}^{p-1}(-1)^iX^iY^{p-i}/i
\Bigr)
\pmod{(X^p,Y^p)}
\]
in the polynomial ring $\F_p[X,Y]$.
\end{prop}

\begin{proof}
We start with
\begin{equation*}
\begin{split}
E(X)\cdot E(Y)-E(X+Y)
&=
\sum_{k=p}^{2p-2}
\sum_{s=k-p+1}^{p-1}\frac{X^sY^{k-s}}{s!\,(k-s)!}
\\
&=
\sum_{i=1}^{p-1} \,
\sum_{k=p}^{2p-2}
\sum_{s=k-p+1}^{p-1}
\binom{s}{i}\binom{k-s}{p-i}
\frac{X^sY^{k-s}}{s!\,(k-s)!},
\end{split}
\end{equation*}
where we have used the fact that for $k$ and $s$ in the given
ranges we have
\[
\sum_{i=1}^{p-1}
\binom{s}{i}\binom{k-s}{p-i}=
\sum_{i=0}^{p}
\binom{s}{i}\binom{k-s}{p-i}=
\binom{k}{p}\equiv 1\pmod{p}.
\]
Now we have
$E(X+Y)
=\sum_{h=0}^{p-1}(X+Y)^h/h!
=\sum_{h=0}^{p-1}
\sum_{j}\binom{h}{j}X^jY^{p-j}/h!
$
where we need not specify the range for the summation variable $j$
because the binomial coefficient vanishes outside that range.
Now we expand
\begin{equation*}
\begin{split}
E(X+Y)
\sum_{i=1}^{p-1}\frac{X^iY^{p-i}}{i!\,(p-i)!}
&=
\sum_{i=1}^{p-1}
\sum_{h=0}^{p-1}
\sum_j
\binom{h}{j}
\frac{X^{i+j}Y^{p-i+h-j}}{h!\,i!\,(p-i)!}
\\
&=
\sum_{i=1}^{p-1}
\sum_{k=p}^{2p-1}
\sum_s
\binom{k-p}{s-i}
\frac{X^{s}Y^{k-s}}{(k-p)!\,i!\,(p-i)!}
\\
&=
\sum_{i=1}^{p-1} \,
\sum_{k=p}^{2p-1}
\sum_{s=0}^k
\binom{s}{i}\binom{k-s}{p-i}
\frac{X^sY^{k-s}}{s!\,(k-s)!}.
\end{split}
\end{equation*}
This is similar to the expression we found earlier for
$E(X)\cdot E(Y)-E(X+Y)$,
except that the summation range for $k$ and $s$ are larger.
However, each additional monomial in the expression found for
$E(X+Y) \sum_{i=1}^{p-1}X^iY^{p-i}/(i!\,(p-i)!)$
is a multiple of either $X^p$ or $Y^p$,
proving that the claimed congruence holds.
\end{proof}

Equation~\eqref{eq:my-obstruction} has a remarkable consequence.
Roughly speaking, although $E(D)$ may fail to be an algebra automorphism,
according to Equation~\eqref{eq:my-obstruction} the location of this failure when $D^p=0$ is somehow more under control
than it would appear from Equation~\eqref{eq:obstruction}.
The crucial point is that in each term of the sum which appears in the right-hand side of Equation~\eqref{eq:my-obstruction}
the derivation $D$ formally appears $p$ times.
(In fact, a multiple of $p$ times is all we need.)
This allows one to show that one feature of automorphisms, that of sending any grading of $A$ into
another grading, is also enjoyed by $E(D)$ to some extent.

\begin{theorem}[{\cite[Theorem~2.3]{Mat:Artin-Hasse}}]\label{thm:truncated-grading}
Let $A=\bigoplus_i A_i$ be a non-associative algebra
over a field of prime characteristic $p$,
graded over the integers modulo $m$.
Suppose that $A$ has a graded derivation $D$ of degree $d$,
with $m\mid pd$, such that $D^p=0$.
Then the direct sum decomposition $A=\bigoplus_i E(D) A_i$
is a grading of $A$ over the integers modulo $m$.
\end{theorem}

It turns out that the assumption $D^p=0$ of Theorem~\ref{thm:truncated-grading} can be relaxed to a bare nilpotency assumption
provided we replace the truncated exponential with the Artin-Hasse exponential.
The Artin-Hasse exponential series is defined as
\[
E_p(X):=
\exp\left(\sum_{i=0}^{\infty}X^{p^i}/p^i\right)=
\prod_{i=0}^{\infty}\exp(X^{p^i}/p^i).
\]
The infinite product makes sense because only a finite number of factors
are needed to compute the coefficient of a given power of $X$ in the result.
The Artin-Hasse exponential is a formal power series in $\Q[[X]]$, but its coefficients
are $p$-adic integers, and so the series actually belongs to
$\Z_{(p)}[[X]]$, where $\Z_{(p)}$ is the localization of $\Z$
at the complement of the ideal $(p)$.
This is an immediate application of the Dieudonn\'{e}-Dwork criterion,
see for example~\cite[p.~93]{Koblitz} or~\cite[p.~392]{Robert}.
In particular, the Artin-Hasse exponential can be evaluated on nilpotent elements of any ring
of characteristic $p$.
The congruence of Proposition~\ref{prop:functional-truncated} for the truncated exponential admits the following analogue
for the Artin-Hasse series, viewed modulo $p$
(see the proof of~\cite[Theorem~2.2]{Mat:Artin-Hasse}).

\begin{prop}\label{prop:functional-AH}
There exist
$a_{ij}\in\F_p$ with $a_{ij}=0$ unless $p\mid i+j$, such that
\[
E_p(X)\cdot E_p(Y)=
E_p(X+Y)
\Bigl(1+
\sum_{i,j=1}^{\infty}a_{ij}X^iY^j
\Bigr)
\]
in the power series ring $\F_p[[X,Y]]$.
\end{prop}

Thus, the quotient series $E_p(X)E_p(Y)/E_p(X+Y)$ has the property that
all terms have degree divisible by $p$.
It was proved in~\cite{Mat:exponential} that this
essentially characterizes the reduction modulo $p$ of the
Artin-Hasse series, up to some natural variations
(see the final part of the proof of our Proposition~\ref{prop:S(X)}).
The consequence for gradings is the following generalization of Theorem~\ref{thm:truncated-grading}.

\begin{theorem}[{\cite[Theorem~1]{Mat:Artin-Hasse}}]\label{thm:AH-grading}
Let $A=\bigoplus_i A_i$ be a non-associative algebra
over a field of prime characteristic $p$,
graded over the integers modulo $m$.
Suppose that $A$ has a nilpotent graded derivation $D$ of degree $d$,
with $m\mid pd$.
Then the direct sum decomposition $A=\bigoplus_i E_p(D) A_i$
is a grading of $A$ over the integers modulo $m$.
\end{theorem}

Note that the derivation $D$ is graded of the same degree $d$ also with respect to the new grading
given by Theorem~\ref{thm:AH-grading}, because it commutes with $E_p(D)$.
Also, Theorem 4 finds applications also to cyclic gradings where the
condition $m \vert pd$ is originally not satisfied (and even in the case $m=0$),
as long as we are willing to pass to a coarser grading where the degrees
are viewed modulo the greatest common divisor $(m,pd)$.
Clearly, $D$ is also a graded derivation of degree $d$
(viewed modulo $(m,pd)$) with respect to the latter grading, and Theorem~\ref{thm:AH-grading} applies.

Theorem~\ref{thm:AH-grading} may be described as an instance of a {\em grading switching,}
where the name is inspired by the technique of~\emph{toral switching} in modular Lie algebras.
We briefly discuss the connection with the latter, referring the reader to~\cite[Section~3]{Mat:Artin-Hasse}
for further details.
Roughly speaking, toral switching replaces a torus $T$ of a restricted Lie algebra $L$
with another torus $T_x$ which is more suitable for further study of $L$.
In the simplest and original setting of~\cite{Win:toral} this amounts to applying to $T$ the
exponential of the inner derivation $\ad x$, for some root vector $x\in L$ with respect to $T$.
Because $(\ad x)^2T=0$ the exponential of $\ad x$ can be taken to
be $1+\ad x$ for this purpose.
This is reminiscent of, and certainly motivated by,
the classical characteristic
zero situation where $\exp(\ad x)$ for some root vector $x$ is used to
conjugate a Cartan subalgebra into another.
However, in more general settings $(1+\ad x)T$ fails to be a
torus, and hence the construction of $T_x$ is slightly more involved.
This technique was originally introduced by Winter in~\cite{Win:toral} and later generalized by Block and Wilson
in~\cite{BlWil:rank-two}.
The most general version was finally produced by Premet
in~\cite{Premet:Cartan}.
An exposition of Premet's version can be found in~\cite[Section~1.5]{Strade:book}.
A crucial step in this process is to keep track of the root space decomposition with respect to the new torus,
by constructing certain linear maps $E(x,\lambda)$ (in the notation of~\cite[Section~1.5]{Strade:book})
from the root spaces with respect to $T$ onto the root spaces with respect to $T_x$.
It was shown in~\cite[Section~3]{Mat:Artin-Hasse} that if $x$ is $p$-nilpotent then
$E(x,\lambda)$ coincides with a variation of $E_p(\ad x)$.

Our grading switching, however, is able to produce certain gradings which are not attainable by toral switching,
specifically because they are over groups having elements of order $p^2$,
while gradings obtained as toral decompositions are over groups of exponent $p$.
One concrete example arises from the Zassenhaus algebra $W(1:n)$, which was shown in~\cite[Example~5.3 and Theorem~2]{Mat:Artin-Hasse}
to have gradings {\em genuinely} over the integers modulo $p^{s+1}$, for any $0\le s<n$,
where {\em genuinely} means that they cannot be obtained from gradings over
the integers, or the integers modulo a larger power of $p$,
simply by reducing the degrees modulo $p^{s+1}$.
We review that instance of grading switching in Example~\ref{ex:Zassenhaus},
using more recently established terminology which we recall in Section~\ref{sec:applications}.

In the next section we develop a more general version of grading switching
where the assumption that $D$ is nilpotent is unnecessary.

\section{Laguerre polynomials of derivations}\label{sec:laguerre}

In order to remove the nilpotency assumption on $D$ of Theorem~\ref{thm:AH-grading}
we need to replace the Artin-Hasse series with a new tool,
consisting of certain Laguerre polynomials.
The classical (generalized) Laguerre polynomial of degree $n \geq 0$ is defined as
\[
L_n^{(\alpha)}(X)=\sum_{k=0}^n\binom{\alpha+n}{n-k}
\frac{(-X)^k}{k!},
\]
where $\alpha$ is a parameter, usually taken in the complex numbers.
However, we may also view $L_n^{(\alpha)}(X)$ as a polynomial with rational coefficients in the two
indeterminates $\alpha$ and $X$.

Now fix a prime $p$.
We are essentially interested only in the polynomial
$L_{p-1}^{(\alpha)}(X)$.
The reason is that, viewed in characteristic $p$, it may be thought of as a generalization
of the truncated exponential
$E(X)=\sum_{k=0}^{p-1}X^k/k!$
which we mentioned in the introduction.
In fact, we have
$L_{p-1}^{(0)}(X)\equiv E(X)\pmod{p}$
because
$\binom{p-1}{k}\equiv\binom{-1}{k}=(-1)^k\pmod{p}$
for $0\le k<p$,
and the full sense of this generalization should be conveyed by
the congruence
\begin{equation}\label{eq:Lmodp}
L_{p-1}^{(\alpha)}(X)
\equiv
(1-\alpha^{p-1})
\sum_{k=0}^{p-1}\frac{X^k}
{(\alpha+k)(\alpha+k-1)\cdots(\alpha+1)}
\pmod{p},
\end{equation}
which holds because
$(\alpha+p-1)\cdots(\alpha+1)\equiv \alpha^{p-1}-1\pmod{p}$.

The crucial property of $L_{p-1}^{(\alpha)}(X)$ for our purposes is that its reduction modulo $p$ satisfies the
differential equation
\begin{equation}\label{eq:L-diff}
X\cdot\frac{d}{dX}
L_{p-1}^{(\gamma)}(X)
\equiv
(X-\gamma)\cdot L_{p-1}^{(\gamma)}(X)
+
X^p-(\gamma^p-\gamma)
\pmod{p}.
\end{equation}
In the special case where $\gamma=0$ this reads
\begin{equation}\label{eq:E-diff}
XE'(X)
\equiv
XE(X)+X^p
\pmod{p}
\end{equation}
in terms of the truncated exponential $E(X)$.

The differential equation modulo $p$ for $L_{p-1}^{(\alpha)}(X)$
which we stated in Equation~\eqref{eq:L-diff} is important to us
because it implies a congruence similar to the
functional equation $\exp(X)\exp(Y)=\exp(X+Y)$ satisfied by the classical
exponential.
Before stating our precise result in Proposition~\ref{prop:functional-Laguerre}
we pause to recall how this might be done in the classical case of the ordinary exponential in characteristic zero.
Suppose that a power series $F(X)\in 1+X\Q[[X]]$ satisfies the differential equation $F'(X)=F(X)$
of the exponential function.
Because the series $F(X+Y)$ is invertible in $\Q[[X,Y]]$ we may consider the quotient $F(X)F(Y)/F(X+Y)$.
Its partial derivative with respect to either variable
is seen to vanish because of the differential equation, and hence the quotient series
must be a constant, necessarily equal to $1$.
A similar argument, although technically more involved,
yields the following result.

\begin{prop}[{\cite[Proposition~2]{AviMat:Laguerre}}]\label{prop:functional-Laguerre}
Consider the subring
$R=\F_p\bigl[\alpha,\beta,\bigl((\alpha+\beta)^{p-1}-1\bigr)^{-1}\bigr]$ of the ring
$\F_p(\alpha,\beta)$ of rational expressions in the indeterminates $\alpha$ and $\beta$,
and let $X$ and $Y$ be further indeterminates.
There exist rational expressions $c_i(\alpha,\beta)\in R$, such that
\[
L_{p-1}^{(\alpha)}(X)
L_{p-1}^{(\beta)}(Y)
\equiv
L_{p-1}^{(\alpha+\beta)}(X+Y)
\Bigl(
c_0(\alpha,\beta)+\sum_{i=1}^{p-1}c_i(\alpha,\beta)X^iY^{p-i}
\Bigr)
\]
in $R[X,Y]$, modulo the ideal generated by
$X^p-(\alpha^p-\alpha)$
and
$Y^p-(\beta^p-\beta)$.
\end{prop}

As was the case with Propositions~\ref{prop:functional-truncated} and~\ref{prop:functional-AH},
for applications to gradings the key property of the congruence in Proposition~\ref{prop:functional-Laguerre} is that the polynomial
$c_0(\alpha,\beta)+\sum_{i=1}^{p-1}c_i(\alpha,\beta)X^iY^{p-i}$
has only terms of total degree multiple of $p$.

Because $L_{p-1}^{(0)}(X)$
coincides modulo $p$ with
$E(X)$,
Proposition~\ref{prop:functional-truncated} is recovered from
Proposition~\ref{prop:functional-Laguerre} by setting $\alpha=\beta=0$.
Proposition~\ref{prop:functional-AH} may also be viewed as a variation of the special case
of Proposition~\ref{prop:functional-Laguerre} obtained by setting
$\alpha=-\sum_{i=1}^{\infty}X^{p^i}$ and
$\beta=-\sum_{i=1}^{\infty}Y^{p^i}$.
The proper interpretation of this statement will be clarified by the following result.

\begin{prop}\label{prop:S(X)}
The power series
$S(X)=L_{p-1}^{(-\sum_{i=1}^{\infty}X^{p^i})}(X)\in\F_p[[X]]$
satisfies
\[
S(X)=E_p(X)\cdot G(X^p)
\]
for some $G(X)\in 1+X\F_p[[X]]$.
\end{prop}

Note that the power series $S(X)$ defined in Proposition~\ref{prop:S(X)} is the same as that
defined in~\cite[Section~3]{Mat:Artin-Hasse}.

\begin{proof}
Informally, we want to set
$\alpha=-\sum_{i=1}^{\infty}X^{p^i}$ and
$\beta=-\sum_{i=1}^{\infty}Y^{p^i}$
in Proposition~\ref{prop:functional-Laguerre}.
To put this on solid ground, start viewing the congruence stated in Proposition~\ref{prop:functional-Laguerre}
as an equality in the quotient ring $R[X,Y]/\mathcal{I}$,
where $\mathcal{I}$ is the ideal of the polynomial ring $R[X,Y]$ generated by
$X^p-(\alpha^p-\alpha)$
and
$Y^p-(\beta^p-\beta)$.
The ring $R$ may be viewed itself as the quotient ring of the polynomial ring
$\F_p[\alpha,\beta,\gamma]$ modulo the ideal generated by $\gamma\bigl((\alpha+\beta)^{p-1}-1\bigr)-1$.
The unique ring morphism of $\F_p[\alpha,\beta,\gamma]$ into the power series ring $\F_p[[X,Y]]$
which sends $\alpha$ to $-\sum_{i=1}^{\infty}X^{p^i}$,
$\beta$ to $-\sum_{i=1}^{\infty}Y^{p^i}$, and
$\gamma$ to $\bigl(-1+(\sum_{i=1}^{\infty}X^{p^i}+Y^{p^i}\bigr)^{p-1})^{-1}$,
induces a ring morphism of $R$ into $\F_p[[X,Y]]$.
In turn, this extends uniquely to a morphism of the polynomial ring $R[X,Y]$ into $\F_p[[X,Y]]$
sending $X$ to $X$ and $Y$ to $Y$.
Finally, this sends the ideal $\mathcal{I}$ to zero, and hence induces a morphism of
$R[X,Y]/\mathcal{I}$ to $\F[[X,Y]]$.
The result of applying this final morphism to the congruence stated in Proposition~\ref{prop:functional-Laguerre}
reads
\[
S(X)
S(Y)
=
S(X+Y)
\Bigl(
c'_0(X^p,Y^p)+\sum_{i=1}^{p-1}c'_i(X^p,Y^p)X^iY^{p-i}
\Bigr),
\]
where $c'_i(X^p,Y^p)=c_i\bigl(-\sum_{i=1}^{\infty}X^{p^i},-\sum_{j=1}^{\infty}Y^{p^j}\bigr)$.
Consequently, the series $S(X)\in 1+X\F_p[[X]]$ has the property that
$S(X)S(Y)/S(X+Y)$
has only terms of terms of degree divisible by $p$.
According to the Theorem in~\cite{Mat:exponential},
this is equivalent to $S(X)$ being of the form
$S(X)=E_p(cX)\cdot G(X^p)$,
for some $c\in\F_p$ and $G(X)\in 1+X\F_p[[X]]$.
Because $S(X)\equiv L^{(0)}_{p-1}(X)=E(X)\equiv E_p(X)\pmod{X^p}$,
the constant $c$ can only be $1$, and the desired conclusion follows.
\end{proof}

Proposition~\ref{prop:S(X)} and its proof explain how Proposition~\ref{prop:functional-AH}
may be viewed as a special case of Proposition~\ref{prop:functional-Laguerre}.
In a similar fashion as in the results of the previous section,
Proposition~\ref{prop:functional-Laguerre} allows one to prove a generalization
of Theorem~\ref{thm:AH-grading} where $D$ is not assumed to be nilpotent.
We state here only a special version of our result, where the assumption that all eigenvalues of $D$ belong to the prime field $\F_p$
produces a simpler statement.
This version is sufficient to perform an explicit instance of grading switching on an Albert-Zassenhaus algebra which we describe in Example~\ref{ex:AZ}.

\begin{theorem}[{a special case of~\cite[Theorem~4]{AviMat:Laguerre}}]\label{thm:Laguerre-grading}
Let $A=\bigoplus_k A_k$ be a non-associative algebra over the field $\F$ of characteristic $p>0$,
graded over the integers modulo $m$.
Suppose that $A$ has a graded derivation $D$ of degree $d$
such that  $D^{p^{r+1}}=D^{p^r}$, with $m\mid pd$.
Suppose that $\F$ contains the field of $p^p$ elements, and choose $\gamma \in \F$ with
$\gamma^p-\gamma=1$.
Let $A=\bigoplus_{a\in\F_p}A^{(a)}$ be the decomposition of $A$
into a direct sum of generalized eigenspaces for $D$.
Set $h(t)=\sum_{i=1}^{r-1}t^{p^i}$, and
let $\LL_D:A\to A$ be the linear map whose restriction to
$A^{(a)}$ coincides with $L_{p-1}^{(a\gamma-h(D))}(D)$.
Then
$A=\bigoplus_k\LL_D(A_k)$ is also a grading of $A$ over the integers modulo $m$.
\end{theorem}

Note that on the subalgebra $A^{(0)}$ of $A$ our map $\LL_D$ coincides with $S(D)$,
where $S(X)$ is as defined in Proposition~\ref{prop:S(X)}.
Consequently, the grading $A^{(0)}=\bigoplus_k\LL_D(A_k\cap A^{(0)})$
of that subalgebra is the same as the one we would obtain by an application to $A^{(0)}$ of Theorem~\ref{thm:AH-grading},
which is based on $E_p(D)$.
(The possible extra factor $G(X^p)$ of $S(X)$ which appears in Theorem~\ref{prop:S(X)} is immaterial here,
because $D^p$ is a derivation of degree zero, and hence $G(D^p)$ does not affect the grading.)

We refer the reader to~\cite[Theorem~4]{AviMat:Laguerre} for a generalization of Theorem~\ref{thm:Laguerre-grading}
which, up to possibly extending the field $\F$,
only assumes that some power $D^{p^r}$ is semisimple with finitely many
eigenvalues.
The map $\LL_D$ of Theorem~\ref{thm:Laguerre-grading} extends to that more general situation.
When specialized to the toral switching setting of~\cite[Section~1.5]{Strade:book},
which we recalled after Theorem~\ref{thm:AH-grading},
our map $\LL_{\ad x}$ coincides with the map $E(x,\lambda)$
which connects root spaces of the old and new torus.
In fact, our construction of the map $\LL_D$ was strongly inspired by toral switching.
However, as was the case for the grading switching by means of Artin-Hasse exponentials,
our more general grading switching based on Laguerre polynomials can be used to obtain gradings
which are not achievable by toral switching,
as they are over groups with elements of order $p^2$.
We illustrate this assertion with Example~\ref{ex:AZ}.

\section{Applications}\label{sec:applications}

The {\em grading switching} described in our increasingly more general
Theorems~\ref{thm:truncated-grading}, \ref{thm:AH-grading}
and~\ref{thm:Laguerre-grading} was initially motivated by
the need to produce certain gradings in the rather specialized
area of {\em thin Lie algebras},
but is capable of applications of more general interest.
In this section we review one from~\cite{Mat:Artin-Hasse}
which concerns cyclic gradings of Zassenhaus algebras,
and then present a new application to cyclic gradings
of Albert-Zassenhaus algebras.
Some {\em ad hoc} terminology was used in the former example in~\cite{Mat:Artin-Hasse},
where the cyclic gradings of a Zassenhaus algebra produced there
were said to be {\em genuinely} over some cyclic group $\Z/p^r\Z$,
meaning that they cannot be obtained from any $(\Z/m\Z)$-grading
with $m=0$ or a power of $p$ greater than $p^r$ by viewing the degrees modulo $p^r$.
We will recast that statement using more general
definitions pertaining to gradings, as in~\cite{Kochetov:gradings,EldKoc:book},
which we recall only in the amount of generality that we actually need.

A $G$-grading $\Gamma$ of a non-associative algebra $A$, where $G$ is an abelian group,
is a vector space decomposition $\Gamma:A=\bigoplus_{g\in G}A_g$,
such that $A_{g_1}A_{g_2}\subseteq A_{g_1+g_2}$ for each $g_1,g_2\in G$.
The {\em support} of the grading is the subset $S=\{g\in G\mid A_g\neq 0\}$.
It is clearly not restrictive to assume that $S$ generates $G$, and so we do that in the sequel.
Our reason for writing the operation in $G$ in additive notation
is that any group grading of a simple Lie algebra, the case of interest here,
is actually an abelian group grading, see~\cite[Proposition~3.3]{Kochetov:gradings}.

A $G$-grading $\Gamma:A=\bigoplus_{g\in G}A_s$
is said to be a {\em refinement} of the $G'$-grading $\Gamma':A=\bigoplus_{g'\in G'}A'_{g'}$
(and $\Gamma'$ is a {\em coarsening} of $\Gamma$)
if each $A_g$ is contained in some $A'_{g'}$.
The refinement is {\em proper} if this containment is strict in at least one case.
We will say that a grading is {\em fine} it if does not admit any proper refinement
(in the class of group gradings, as considered here).

A more general definition of a grading $\Gamma$ is a vector space decomposition $\Gamma:A=\bigoplus_{s\in S}A_s$,
with $A_s\neq 0$ for all $s\in S$,
where $S$ is just a set, such that for each $s_1,s_2\in S$ we have
$A_{s_1}A_{s_2}\subseteq A_{s_3}$ for some $s_3\in S$.
Not every grading $\Gamma$ in this generality may be {\em realized} as group grading,
which means embedding the $S$ as a subset of a group $G$ in a way to turn the grading
into a $G$-grading, but if it does then there exists a {\em universal group} for $\Gamma$.
We will bypass some details of the usual definition by using an equivalent definition which is sufficient for our purposes.
(The equivalence follows from~\cite[Proposition~3.15]{Kochetov:gradings}.)

For a $G$-grading $\Gamma:A=\bigoplus_{g\in G}A_g$,
we say that $G$ is the {\em universal (grading) group} for that grading if the following holds:
given any group grading $\Gamma':A=\bigoplus_{h\in H}A'_h$, for some abelian group $H$,
which is a coarsening of $\Gamma$,
there exists a unique group homomorphism
$f:G\to H$ such that $A_g\subseteq A'_{f(g)}$, for all $g\in G$
(whence $A'_h=\bigoplus_{g\in f^{-1}(h)}A_g$).
Note that the restriction of $f$ to the support of $\Gamma$ is uniquely determined by this requirement.

Two gradings $\Gamma:A=\bigoplus_{s\in S}A_s$ and $\Gamma':A=\bigoplus_{s'\in S'}A'_{s'}$ are said to be {\em equivalent}
if there exist an algebra automorphism $\varphi:A\to A$
and a bijection $\alpha:S\to S'$ such that
$\varphi(A_s)=A'_{\alpha(s)}$
for all $s\in S$.
According to~\cite[Proposition~3.7]{Kochetov:gradings},
when those are group gradings over their universal groups,
it is not restrictive to require that $\alpha$ is a group homomorphism.

As a simple illustration of the above concepts, consider the Zassenhaus algebra $W(1;n)$,
whose definition we recall now.
The algebra $\F[x;n]$ of \emph{divided powers} in one indeterminate $x$ of height $n$,
over a field $\F$ of prime characteristic $p$,
is the associative $\F$-algebra with basis elements $x^{(i)}$, for $0\leq i<p^{n}$,
and multiplication defined by
$x^{(i)}\cdot x^{(j)}=\binom{i+j}{i} x^{(i+j)}$.
The Zassenhaus algebra $W(1;n)$ can be defined as the Lie subalgebra of $\Der(\F[x;n])$
consisting of the {\em special} derivations $f\partial$, with $f\in\F[x;n]$.

We prove that $\Z$ is the universal group for the (standard) $\Z$-grading
$W(1;n)=\bigoplus_{i\in\Z}A_i$, where $A_i$ is spanned by $x^{(i+1)}\partial$ for $-1\le i<p^n-1$,
and is zero otherwise.
Consider a group grading $A=\bigoplus_{h\in H}A'_h$, for some abelian group $H$,
such that each $A_i$ is contained in some $A'_h$.
For $-1\le i<p^n-1$, let $f(i)$ be the unique element of $H$ such that $A_i\subseteq A'_{f(i)}$.
Because $[\partial,x^{(i+1)}\partial]=x^{(i)}\partial$
we have $0\neq A'_{f(i-1)}\subseteq A'_{f(-1)+f(i)}$, whence
$f(i-1)=f(-1)+f(i)$ for $-1<i<p^n-1$.
It follows inductively that $f(i)=-i\cdot f(-1)$ for $-1\le i<p^n-1$.
Hence this partially defined function $f$ extends in a unique way to a homomorphism $f:\Z\to H$, by setting
$f(i)=-i\cdot f(-1)$ for all $i\in\Z$, and this extension clearly satisfies the required property
$A'_h=\bigoplus_{i\in f^{-1}(h)}A_i$.

We take a less trivial example from~\cite[Theorem~2]{Mat:Artin-Hasse}, where a grading of $W(1;n)$
over $\Z/p^{s+1}\Z$ was produced, with $0\le s<n$, for which $\Z/p^{s+1}\Z$ is the universal group.
That construction involves a grading switching in the special form of Theorem~\ref{thm:AH-grading}.

\begin{example}[{\cite[Example~5.3 and Theorem~2]{Mat:Artin-Hasse}}]\label{ex:Zassenhaus}
Consider a Zassenhaus algebra $W(1;n)$ in characteristic $p>3$, and let $0\le s<n$.
Then the $p^s$-th power $D=(\ad\partial)^{p^s}$
of the inner derivation $\ad\partial$ of $W(1;n)$ is graded of degree $-p^s$ and satisfies $D^{p^{n-s}}=0$.
We are in a position to apply Theorem~\ref{thm:AH-grading}
to the $(\Z/p^{s+1}\Z)$-grading of $W(1;n)$ obtained from the $\Z$-grading by viewing the degrees modulo $p^{s+1}$.
Thus, $E_p(D)$ sends this grading into another $(\Z/p^{s+1}\Z)$-grading of $W(1;n)$.
We now prove that $\Z/p^{s+1}\Z$ is the universal group for this grading.

We start with with proving the conclusion in the special case where $n=s+1$.
Thus, our grading is
$\Gamma:W(1;s+1)=\bigoplus_{i\in\Z/p^{s+1}\Z}A_i$, where $A_i$ is spanned by $E_p(D)(x^{(i+1)}\partial)$ for $-1\le i<p^{s+1}-1$.
(Note that to avoid cumbersome notation we are using the same letter $i$ both for an integer in the range considered
and for its residue class modulo $p^{s+1}$.)
Consider a group grading $\Gamma':W(1;s+1)=\bigoplus_{h\in H}A'_h$, for some abelian group $H$,
such that each $A_i$ is contained in some $A'_h$.
For $-1\le i<p^{s+1}-1$, let $f(i)$ be the unique element of $H$ such that $A_i\subseteq A'_{f(i)}$.
As shown in~\cite[Example~5.3]{Mat:Artin-Hasse} we have
\[
[E_p(D)\partial,E_p(D)(x^{(i+1)}\partial)]=
\begin{cases}
\bigl((i+1)/p^s\bigr)\cdot E_p(D)(x^{(i)}\partial)
&\text{if $p^s\mid i+1$,}
\\
E_p(D)(x^{(i)}\partial)
&\text{otherwise.}
\end{cases}
\]
This yields $0\neq A'_{f(i-1)}\subseteq A'_{f(-1)+f(i)}$, whence
$f(i-1)=f(-1)+f(i)$ for $-1<i<p^{s+1}-1$,
and it follows inductively that $f(i)=-i\cdot f(-1)$ for $-1\le i<p^{s+1}-1$.
Moreover,
$[E_p(D)(x^{(2p^s)}\partial),E_p(D)(x^{((p-1)p^s)}\partial)]=(p-3)E_p(D)(x^{(p^s-1)}\partial)$
yields
$f(p^s-2)=f(2p^s-1)+f((p-1)p^s-1)$,
because $p>3$,
whence
$p^{s+1}\cdot f(-1)=0$.
Hence $f$, which is clearly the unique function satisfying the required property
$A'_h=\bigoplus_{i\in f^{-1}(h)}A_i$,
is a group homomorphism $f:\Z/p^{s+1}\Z\to H$,
as desired.
Note that the grading $\Gamma$ is a fine grading of $W(1;s+1)$, because
all its homogeneous components are one-dimensional.
In particular, it is not equivalent to the standard $\Z$-grading, because
the universal grading group is different in the two cases.

Now we deal with the general case $n\ge s+1$.
Let
$W(1;n)=\bigoplus_{i\in\Z/p^{s+1}\Z}A_i$ be the grading produced by the grading switching as described earlier.
Note that the subalgebra $W(1;s+1)$ of $W(1;n)$, being normalized by $D$, is also
a graded subalgebra with respect to this grading.
In fact, $W(1;s+1)=\bigoplus_{i\in\Z/p^{s+1}\Z}A''_i$, where $A''_i=A_i\cap W(1;s+1)$
coincides with what we denoted by $A_i$ in the previous paragraph.
Let $W(1;n)=\bigoplus_{h\in H}A'_h$ be another group grading, for some abelian group $H$,
such that each $A_i$ is contained in some $A'_h$.
Now $W(1;s+1)$ is a graded subalgebra of $W(1;n)$ with respect to its $H$-grading as well, because
$\sum_{h\in H}\left(A'_h\cap W(1;s+1)\right)=
\sum_{h\in H}\left(A'_h\cap\sum_{i}A''_i\right)=
\sum_{h\in H}\sum_i\left(A'_h\cap A''_i\right)=
\sum_i\left(A''_i\cap\sum_{h\in H}A'_h\right)=
\sum_iA''_i=W(1;s+1)$.
According to the definition of universal group,
there is a unique homomorphism $f:\Z/p^{s+1}\Z\to H$ such that
$A_i\cap W(1;s+1)\subseteq A'_{f(i)}\cap W(1;s+1)$ for all $i\in\Z/p^{s+1}\Z$.
Because $A_i\cap W(1;s+1)\neq 0$ for all $i$, and each $A_i$ is contained in a unique $A'_h$
it follows that $A_i\subseteq A'_{f(i)}$ for all $i$, as desired.

As explained in~\cite[Remark~5.5]{Mat:Artin-Hasse},
the grading switching performed in this example is not attainable by toral switching,
except when $n=1$.
In that case $W(1;n)$ is restricted, and the
$(\Z/p\Z)$-grading obtained from the standard grading by reducing the degrees modulo $p$ coincides with the
root space decomposition with respect to the torus spanned by $x\partial$.
Toral switching with respect to the element $\partial$ of $W(1;1)$
produces the torus spanned by $E_p(D)(x\partial)=(1+x)\partial$, and the corresponding
root spaces can be obtained by applying $E_p(D)$ to the original root spaces.
\end{example}

Our second example plays a role in~\cite{AviMat:mixed_types}, where a variation of a grading we describe here
can be realized as a {\em thin grading} of an Albert-Zassenhaus algebra.
In this example the derivation involved in the grading switching is not nilpotent,
and so the grading switching of Theorem~\ref{thm:AH-grading} based on the Artin-Hasse series is not sufficient,
but we need the more general version with Laguerre polynomials.
Our Theorem~\ref{thm:Laguerre-grading} will be general enough because the derivation involved
has all its eigenvalues in the prime field.

\begin{example}\label{ex:AZ}
The algebra
$\F[x,y;n,m]$ of divided powers in two
indeterminates $x$ and $y$ of heights $n$ and $m$, over a field $\F$ of prime characteristic $p$,
may be identified with the tensor product algebra
$\F[x;n]\otimes\F[y;m]$.
Thus, a basis is given by the monomials
$x^{(i)}y^{(j)}$, for
$0\leq i<p^{n_1}$ and $0\leq j<p^{n_2}$,
which multiply according to the rule
$x^{(i)} y^{(j)}x^{(k)}y^{(l)}=\binom{i+k}{i}\binom{j+l}{j} x^{(i+k)}y^{(j+l)}$.
We use the
standard shorthands $\bar{x}=x^{(p^n-1)}$ and
$\bar{y}=y^{(p^m-1)}$.
The Albert-Zassenhaus algebra
$H(2;(n,m);\Phi(1))$, which is simple if $p>2$, can be identified with the vector space
$\F[x,y;n,m]$ endowed with the Lie bracket
(a {\em Poisson bracket})
\begin{align}
&\{x^{(i)}
y^{(j)},
x^{(k)}
y^{(l)}\}
=
N(i,j,k,l)\,
x^{(i+k-1)}
y^{(j+l-1)}
\quad\text{if $i+k>0$, and}\label{eq:Poisson_1}
\\
&\{y^{(j)},
y^{(l)}\}
=
\left(
\binom{j+l-1}{l}-
\binom{j+l-1}{j}
\right)\bar x
y^{(j+l-1)},\label{eq:Poisson_exception}
\end{align}
where we have set
\[
N(i,j,k,l):=
\binom{i+k-1}{i}
\binom{j+l-1}{j-1}-
\binom{i+k-1}{i-1}
\binom{j+l-1}{j}.
\]
See~\cite[Chapter~6]{Strade:book} and~\cite{AviMat:A-Z} for this and further details.
Equations~\eqref{eq:Poisson_1} and~\eqref{eq:Poisson_exception} show that
$H(2;(n,m);\Phi(1))$
is graded over the group $\Z/p^n\Z\times\Z$
by assigning degree $(i+p^n\Z,j)$ to the monomial $x^{(i+1)}y^{(j+1)}$.
It is not difficult to see that $\Z/p^n\Z\times\Z$ is the universal group for this grading.

Assuming $p>3$, consider the Albert-Zassenhaus algebra $A=H(2;(n,m);\Phi(1))$,
for some $n,m>0$, and its derivation $D=(\ad y)^{p^s}$, for some $0\le s<n$.
Writing each monomial in $\F[x,y;n,m]$
in the form  $x^{(ap^s)} x^{(k+1)}y^{(j+1)}$,
with $0\le a<p^{n-s}$, $-1\leq k<p^s-1$ and $-1\leq j<p^m-1$,
we have
\[
D(x^{(ap^s)} x^{(k+1)}y^{(j+1)})=
\begin{cases}
x^{((a-1)p^s)}x^{(k+1)}y^{(j+1)}& \textrm{if $a>0$,}\\
-jx^{(p^{n}-p^s)}x^{(k+1)}y^{(j+1)}& \textrm{if $a=0$.}
\end{cases}
\]
Consequently, $D^{p^{n-s}}$ acts semisimply on $A$, with eigenvalues in the prime field, as
\[
D^{p^{n-s}}(x^{(ap^s)} x^{(k+1)}y^{(j+1)})=-jx^{(ap^s)}
x^{(k+1)}y^{(j+1)},
\]
and hence $D^{p^{n-s+1}}=D^{p^{n-s}}$.
In particular, $D$ has all its eigenvalues in the prime field.

Consider the cyclic grading of $A=\bigoplus A^{x}_{ap^s+k}$ over the group $\Z/p^{s+1}\Z$,
obtained by assigning degree $ap^s+k+p^{s+1}\Z$ to the
monomial $x^{(ap^s)}x^{(k+1)}y^{(j+1)}$
(and the superscript $x$ in the homogeneous components reminds us of this choice of a degree).
By assigning that monomial degree $j$ instead, we obtain a grading $\Z$-grading $A=\bigoplus A^y_j$,
whose zero-component $A^y_0$ is isomorphic with the Zassenhaus algebra $W(1;n)$.
These two gradings together determine a grading $A=\bigoplus(A^x_{ap^s+k}\cap A^y_j)$
over $\Z/p^{s+1}\Z \times \Z$.
Now $D$ is a homogeneous derivation of degree $d=-p^s+p^{s+1}\Z$ with respect to the former grading,
whence $pd$ vanishes in the grading group.
We apply the grading switching from~\cite[Theorem~4]{AviMat:Laguerre} to this grading,
in the special form recalled here as Theorem~\ref{thm:Laguerre-grading},
and obtain another cyclic grading of $A$, namely,
$A=\bigoplus\mathcal{L}_D(A^x_{ap^s+k})$.
Now intersecting this grading
with $A=\bigoplus A^y_j$ yields a grading
$A=\bigoplus_{(i,j)\in\Z/p^{s+1}\Z\times\Z}A_{(i,j)}$,
where
$A_{(ap^s+k,j)}
=\mathcal{L}_D(A^x_{ap^s+k}\cap A^y_j)
=\mathcal{L}_D(A^x_{ap^s+k})\cap A^y_j
$.
Now we prove that $\Z/p^{s+1}\Z\times\Z$ is the universal group for this grading,
using Example~\ref{ex:Zassenhaus}.

Consider a group grading $A=\bigoplus_{h\in H}A'_h$, over some abelian group $H$,
such that each $A_{(i,j)}$ is contained in some $A'_h$.
For $i\in\Z/p^{s+1}\Z$ and $-1\le j<p^m-1$, let $f(i,j)$ be the unique element of $H$ such that $A_{(i,j)}\subseteq A'_{f(i,j)}$.
Note that the centralizer of the element $\mathcal{L}_D(x)$ of $A_{(0,-1)}$ is contained in $A^y_{-1}$ (and actually equals $A^y_{-1}$),
and hence $[\mathcal{L}_D(x),A_{(i,j)}]\neq 0$ provided $j\ge 0$.
As a consequence, for each $i\in\Z/p^{s+1}\Z$ we have $0\neq A'_{f(i,j-1)}\subseteq A'_{f(i,-1)+f(i,j)}$, whence
$f(i,j-1)=f(i,0)+f(i,j)$ for $-1<j<p^m-1$.
It follows inductively that $f(i,j)=f(i,0)-j\cdot f(i,-1)$ for $-1\le j<p^m-1$, and this is of course for each $i\in\Z/p^{s+1}\Z$.

Now $\sum_{i\in\Z/p^{s+1}\Z}A_{(i,0)}$ coincides with $A^y_0$, which is isomorphic with the Zassenhaus algebra $W(1;m)$ as noted earlier.
Furthermore, its grading $\sum_{i\in\Z/p^{s+1}\Z}A_{(i,0)}$ coincides
with the grading over $\Z/p^{s+1}\Z$ of the Zassenhaus algebra considered in Example~\ref{ex:Zassenhaus},
for which we have proved that $\Z/p^{s+1}\Z$ is a universal group.
Also, $A^y_0$ is a graded subalgebra of $A$ with respect to its $H$-grading as well, because
$\sum_{h\in H}\left(A'_h\cap\sum_{i}A_{(i,0)}\right)=
\sum_{h\in H}\sum_{i}\left(A'_h\cap A_{(i,0)}\right)=
\sum_{i}\left(A_{(i,0)}\cap\sum_{h\in H}A'_h\right)$.
According to the definition of universal group,
there is a unique homomorphism $\bar f:\Z/p^{s+1}\Z\times\{0\}\to H$ such that
$A'_h\cap Z=\bigoplus_{i\in \bar f^{-1}(h)}A_{(i,0)}$ for all $h\in H$.
However, $\bar f$ must coincide with the restriction of $f$ to $\Z/p^{s+1}\Z\times\{0\}$,
because the support of $\bar f$ is this whole group.

The previous paragraph shows that the map $i\mapsto f(i,0)$
is a group homomorphism from $\Z/p^{s+1}\Z$ to $H$.
Together with the previous arguments,
it follows that extending $f$ to a function from $\Z/p^{s+1}\Z\times\Z$ to $H$ by setting
$f(i,j)=f(i,0)-j\cdot f(i,-1)$ for $i\in\Z/p^{s+1}\Z$ and for all $j\in\Z$ results in a group homomorphism, clearly unique.
It clearly satisfies the required property
$A'_h=\bigoplus_{(i,j)\in f^{-1}(h)}A_{(i,j)}$.
We conclude that $\Z/p^{s+1}\Z\times\Z$ is the universal group for the grading under consideration, as claimed.

Now assume $s+1=n$.
Then our $(\Z/p^n\Z\times\Z)$-grading
$A=\bigoplus_{(i,j)\in\Z/p^n\Z\times\Z}A_{(i,j)}$
is fine, simply because all components are one-dimensional.
We prove that it is not equivalent
to the original $(\Z/p^n\Z\times\Z)$-grading
$A=\bigoplus_{(i,j)\in\Z/p^n\Z\times\Z}A'_{(i,j)}$,
where $A'_{(i,j)}$ is spanned by $x^{(i+1)}y^{(j+1)}$.
Because both gradings are over their universal groups,
if they were equivalent there would be a Lie algebra automorphism $\varphi:A\to A$
and an automorphism $\alpha$ of $\Z/p^n\Z\times\Z$ such that
$\varphi(A_{(i,j)})=A'_{\varphi((i,j))}$
for all $(i,j)\in \Z/p^n\Z\times\Z$.
But then to the the restriction of $\alpha$
to the torsion subgroup
$\Z/p^{s+1}\Z\times\{0\}$
of $\Z/p^n\Z\times\Z$
there would correspond a Lie algebra automorphism of
$\bigoplus_{i\in\Z/p^n\Z}A_{(i,0)}\cong W(1;n)$
sending each $A_{(i,0)}$ to the corresponding $A'_{(i,0)}$.
This is not possible, because these are two non-equivalent gradings of a Zassenhaus algebra,
as pointed out at the end of Example~\ref{ex:Zassenhaus}.

Along the lines of~\cite[Remark~5.5]{Mat:Artin-Hasse}, which pointed out the corresponding fact
for the gradings of the Zassenhaus algebra in~Example~\ref{ex:Zassenhaus},
one can show that the gradings of the Albert-Zassenhaus algebra produced here by grading switching
could not be obtained by toral switching, with the exception of the case $n=1$.
In that case, the $p$-envelope
$A_{[p]}=\Der(A)$
of $A=H(2;(n,m);\Phi(1))$ (see~\cite[Section~7.1]{Strade:book})
has a two-dimensional maximal torus $T$ spanned by $xy$
and the derivation $(\ad y)^p$.
The intersection with $A$ of the corresponding root space decomposition of $A_{[p]}$
may be obtained from the $\Z/p\Z\times\Z$-grading of $A$
introduced at the beginning of this example by reducing all degrees modulo $p$.
Toral switching with respect to the root vector $y$
produces a new torus $\mathcal{L}_{\ad y}(T)$, and the corresponding
root spaces can be obtained by applying $\mathcal{L}_{\ad y}$ to the original root spaces.
Hence in our notation each of the new root spaces (on $A$) has the form
$\bigoplus_{k=0}^{p^{m-1}}A_{(i,j+kp)}$,
for some $-1\le i,j<p-1$.
Thus, in this case, but in this case only, our grading switching can be (essentially)
obtained by traditional toral switching.
\end{example}

\bibliographystyle{amsplain}

\bibliography{References}

\end{document}